\documentclass[11pt,reqno]{amsart}
\usepackage{pbruillard_commands}
\newcommand{\ndiv}{\nmid}
\newcommand{\Stab}{\mathrm{Stab}}

\newcommand{\paren}[1]{\left(#1\right)}
\renewcommand{\diag}{\mathrm{diag}}
\newcommand{\intcat}{\mathrm{int}}
\usepackage[backend=bibtex,maxbibnames=9,sorting=nyt,firstinits=true,eprint=false,url=false]{biblatex}
\renewbibmacro{in:}{}
\bibliography{References.bib}

\begin{document}
\title[ Modular categories of dimension $p^3m$]{Modular categories of dimension $p^3m$ with $m$ square-free}
\date{\today}
\author{Paul Bruillard}
\email{Paul.Bruillard@pnnl.gov}
\author{Julia Yael Plavnik}
\email{julia@math.tamu.edu}
\author{Eric C. Rowell}
\email{rowell@math.tamu.edu}

\thanks{\textit{PNNL Information Release:} PNNL-SA-120943}
\thanks{
The authors thank Yu Tsumura for his comments, and C\'{e}sar
Galindo for useful discussions.
J.Y.P., and E.C.R. were partially supported by NSF grant DMS-1410144.
J.Y.P. would like to thank the Association of Women in Mathematics for the
AWM grant. Some of the results in this paper were obtained while the authors
were at the American Institute of Mathematics during May 2016, participating in
a SQuaRE. We would like to thank AIM for their hospitality and encouragement.
The research in this paper was, in part, conducted under the Laboratory Directed
Research and Development Program at PNNL, a multi-program national laboratory
operated by Battelle for the U.S. Department of Energy.}

\subjclass[2010]{Primary 18D10; Secondary 16T06}

\commby{}

\begin{abstract}
  We give a complete classification of modular categories of dimension $p^3m$
  where $p$ is prime and $m$ is a square-free integer. When $p$
  is odd, all such categories are pointed. For $p=2$ one encounters modular
  categories with the same fusion ring as orthogonal quantum groups at certain
  roots of unity, namely $\SO\paren{2m}_{2}$. We also
  classify the more general class of modular categories
  with the same fusion rules as $\SO\paren{2N}_{2}$ with $N$ odd.
\end{abstract}

\maketitle

\section{Introduction}
  \label{Section: Introduction}
  Weakly integral modular categories, \textit{i.e.}, modular categories of
  integral FP-dimension, arise in a number of settings such as quantum groups at
  certain roots of unity \cite{NR1}, equivariantizations of Tambara-Yamagami
  categories \cite{GNN1}, and by gauging pointed modular categories
  \cite{ACRW1,CGPW1,BBCW1}.
  Several key conjectures characterizing weakly
  integral categories are found in the literature, \textit{e.g.} the Property-F
  Conjecture, \cite[Conjecture 2.3]{NR1}, the Weakly Group-Theoretical
  Conjecture \cite[Question 2]{ENO2}, and the Gauging Conjecture
  \cite[Abstract]{ACRW1}.
  In particular, it is known \cite{RWe1} that the braid group representations
  associated with the weakly integral modular categories $SO(M)_2$ factor over finite groups for all $M$; that is,
  these categories have property $F$.
  
  We will assume throughout that all objects in our categories have positive dimensions.  This is not a significant loss of generality: see Remark \ref{rmkunitary}. 
  In the spirit of \cite{HNW}, we call any modular category with the same
  fusion rules as the $4M$-dimensional modular category $SO(M)_2$ with $M$ even an \textit{even metaplectic} modular
  category.\footnote{In fact they require unitarity.  It is not known if this is a strictly stronger assumption that positivity of dimensions.}   Our main results are summarized as follows:
  \begin{thm}
    If $\mcC$ is a non-pointed modular category of dimension $p^{3}m$ where $p$
    is a prime and $m$ is a square-free integer,
    then, $p=2$ and one of the following is
    true:
    \begin{itemize}
      \item[(i)] $\mcC$ is a Deligne product of an even metaplectic modular
      category of dimension $8\ell$ and a pointed $\mbbZ_k$-cyclic modular
      category with $k$ odd, or
      \item[(ii)] $\mcC$ is the Deligne product of a Semion modular category with
      a modular category of dimension $4m$ (see \cite{BGNPRW1}).
    \end{itemize}
  \end{thm}

  Furthermore, we extend the work of \cite[Theorem 3.1]{BGNPRW1} characterizing
  modular categories of dimension $p^2m$, to include the case $p$ odd: in this
  case such categories are pointed (see \thmref{Theorem: p2m main theorem}).
  In particular, all modular categories of dimension $p^{2}m$ and $p^{3}m$
  satisfy the Weakly Group-Theoretical Conjecture \cite[Question 2]{ENO2} and
  Gauging Conjecture \cite{ACRW1}. Moreover, we obtain the following result
  similar to that of \cite[Theorem 3.1]{ACRW1}:
  \begin{thm}[\thmref{Theorem: Metaplectic main result}]
    If $\mcC$ is an even metaplectic modular category of dimension $8N$, with $N$ an odd
    integer, then $\mcC$ is a gauging of the particle-hole symmetry of a
    $\mbbZ_{2N}$-cyclic modular category. Moreover, there are exactly $2^{r+2}$
    inequivalent even metaplectic modular categories of dimension $8N$ for $N = p_1^{k_1} \cdots
    p_r^{k_r}$, with $p_i$ distinct odd primes.
  \end{thm}

\section{Preliminaries}
  \label{Section: Preliminaries}
  In this section we recall results and notation that are presumably well-known
  to the experts. Further details can be found in \cite{ENO1,DGNO1,BKi}.

  A \textit{premodular} category, $\mcC$, is a braided, balanced, fusion
  category. Throughout we will denote the isomorphism classes of simple objects
  of $\mcC$ by $X_{a}$ ordered such that $X_{0}=\1$ is the unit object.  The set of isomorphism
  classes of simple objects will be denoted by $\Irr\paren{\mcC}$. Duality in
  $\mcC$ introduces an involution on the label set of the simple objects by
  $X_{a}^{*}\cong X_{a^{*}}$. The fusion matrices, $N_{a,b}^{c}=\paren{N_{a}}_{b,c}$
  provide the multiplicity of $X_{c}$ in $X_{a}\otimes X_{b}$. These matrices
  are non-negative integer matrices and thus are subject to the Frobenius-Perron
  Theorem. The Frobenius-Perron eigenvalue of $N_{a}$, $d_{a}$, is called the
  \textit{Frobenius-Perron dimension} of $X_{a}$, or FP-dimension for short. An
  object is said to be \textit{invertible} if its FP-dimension is $1$, and \textit{integral} if its
  FP-dimension is an integer. The invertible and integral simple objects
  generate full fusion subcategories called the \textit{pointed subcategory},
  $\mcC_{\pt}$, and the \textit{integral subcategory}, $\mcC_{\intcat}$. If
  $\mcC=\mcC_{\intcat}$, then $\mcC$ is said to be \textit{integral}. The
  categorical FP-dimension is given by $\FPdim\mcC=\sum_{a}d_{a}^{2}$. If
  $\FPdim\mcC\in\mbbZ$, then $\mcC$ is said to be a \textit{weakly integral
  category}. A weakly integral category that is not integral is called
  \textit{strictly weakly integral}.
 
 \begin{rmk}\label{rmkunitary} Let $\mcC$ be any weakly integral modular category.  By \cite[Prop. 5.4]{ENO1} the underlying braided fusion category $\mcC$ has a unique pivotal (in fact, spherical) structure so that each object in $\mcC$ has positive dimension.  Moreover, the  spherical structures on the braided fusion category underlying any modular category are in 1-1 correspondence with simple objects $a$ with $a^{\otimes 2}\cong\1$ and each spherical structure again yields a modular category  \cite[Lemma 2.4]{BNRW1}. Thus a classification of some class of weakly integral modular categories with positive dimensions can be easily extended to a classification of all modular categories in that class.  This motivates our assumption that the categorical and FP-dimensions coincide.
 \end{rmk}
 

Important data for a premodular category are 
  the $S$-matrix and $T$-matrix, $S=\paren{S_{a,b}}$ and
  $T=\paren{\th_{a}\d_{a,b}}$. These matrices are indexed by the simple objects
  in $\mcC$ and the diagonal entries of the $T$-matrix are referred to as
  \textit{twists}.  These data obey (see \cite{BKi})
  $S_{a,b}=S_{b,a}$, $S_{0,a}=d_{a}$, $\bar{S_{a,b}}=S_{a,b^{*}}$
  and the \textit{balancing} relation
  $\th_{a}\th_{b}S_{a,b}=\sum_{c}N_{a^{*},b}^{c}d_{c}\th_{c}$.

  For $\mcD\subset \mcC$ premodular categories the
  \textit{centralizer} of $\mcD$ in $\mcC$ is denoted by $C_{\mcC}\paren{\mcD}$
  and is generated by the objects $\lcb X\in\mcC\mid S_{X,Y}=d_{X}d_{Y}\;\forall
  Y\in\mcD\rcb$ (see \cite{Brug1}. 
  The category $C_{\mcC}\paren{\mcC}$ is called the \textit{M\"{u}ger center} and is often denoted
  by $\mcC'$. Note that a useful characterization of a simple object being
  outside of the M\"{u}ger center is that its column in the $S$-matrix is
  orthogonal to the first column of $S$.
If
  $\mcC'=\Vec$, then $\mcC$ is said to be a \textit{modular category} whereas if
  $\mcC'=\mcC$, then $\mcC$ is called \textit{symmetric}. Symmetric
  categories are classified in terms of group data:
  \begin{theorem}[\cite{D1}]
    If $\mcC$ is a symmetric fusion category, then there exists a finite group $G$ such that
    $\mcC$ is equivalent to the super-Tannakian category $\Rep(G,z)$ of
    super-representations (i.e. $\mbbZ_2$ graded) of $G$ where $z\in Z(G)$ is a distinguished central element with $z^2=1$
 acting as the parity operator.
  \end{theorem}
  Observe that $\mcC\cong\Rep(G,z)$ with $z=e\in G$ if and only if the $\mbbZ_2$-grading is trivial so that $\Rep(G,z)=\Rep(G)$.  In this case we say $\mcC$ is Tannakian, and otherwise we say it is non-Tannakian.  The smallest non-Tannakian symmetric category is $\Rep(\mbbZ_2,1)$ which we will denote by $\sVec$ when equipped with the (unique) structure of a ribbon category with $\dim(\chi)=1$ for a non-trivial object $\chi$.

 An invertible
  object in $\mcC$ that generates a Tannakian category (\text{i.e.} $\Rep\paren{\mbbZ_{2}}$) is
  referred to as a \textit{boson} while an invertible object in $\mcC$
  generating $\sVec$ is referred to as a \textit{fermion}. By \cite[Lemma 5.4]{M5} we find that
  if $\sVec\subset\mcC'$ for some premodular category $\mcC$, and $\chi$ is the
  generator of $\sVec$, then $\chi\otimes Y\ncong Y$ for all simples $Y$ in
  $\mcC$. Bosons are useful through a process known as de-equivariantization
  which we will discuss shortly.

  The Semion and Ising categories will appear frequently in the sequel in factorizations of modular categories. The Semion categories, denoted by $\Sem$, are modular categories with the same fusion rules as
  $\Rep\paren{\mbbZ_{2}}$, see
  \cite{RSW} for details.  
%
  The Ising categories are rank 3 modular categories with simple objects:
  $\1$, $\psi$ (a fermion), and $\s$ (the Ising anyon) of dimension $\sqrt{2}$.
  The key fusion rules are $\ps^{2}=\1$ and
  $\s^{2}=\1+\ps$; while the modular datum can be found in \cite{RSW}.  There are exactly $8$ inequivalent Ising categories.  They are the smallest examples of \textit{generalized Tambara-Yamagami categories}, \textit{i.e.} non-pointed fusion categories with the property that the tensor product of any two non-invertible simple objects is a direct sum of invertible objects.  In fact, by \cite{Nat} any \textit{modular} generalized Tambara-Yamagami category is a (Deligne) product of an Ising category and a pointed modular category.
Characterizations of the Ising categories can be
  found in \cite{DGNO1,BGNPRW1}. In
  Lemmas \ref{Lemma: Equivariantizations contain Ising},
  \ref{Lemma: Dimension 16 SWI MTC Contains Ising}, and
  \corref{Cor: Dimension 8 SWI Fusion Contains Ising} we find new conditions implying a category must contain an Ising category.

  A \textit{grading} of a fusion category $\mcC$ by a finite group $G$ is a
  decomposition of the category as direct sum $\mcC = \oplus_{g\in G} \mcC_g$,
  where the components are full abelian subcategories of $\mcC$ indexed by the
  elements of $G$, such that the tensor product maps
  $\mcC_g\times \mcC_h$ into $\mcC_{gh}$.  The \textit{trivial component}
  $\mcC_e$ (corresponding to the unit of the group $G$) is a fusion subcategory
  of $\mcC$. The grading is called \textit{faithful} if $\mcC_g \neq 0$, for all
  $g\in G$, and in this case all the components are equidimensional with
  $\abs{G}\dim \mcC_g = \dim\mcC$.
  

  Every fusion category $\mcC$ is faithfully graded by the universal grading group $\mcU\paren{\mcC}$
   and every faithful grading of
  $\mcC$ is a quotient of $\mcU\paren{\mcC}$. Furthermore, the trivial component
  under the universal grading is the \textit{adjoint subcategory} $\mcC_{\ad}$, the full fusion subcategory generated by the objects $X\otimes X^*$, for $X$ simple.
  The universal grading was first studied in \cite{GN2}, and it was shown that
  if $\mcC$ is modular, then $\mcU\paren{\mcC}$ is canonically isomorphic to the
  character group of the group $G(\mcC)$ of isomorphism classes of invertible objects in
  $\mcC$. 
  
  When $\mcC$ is a weakly integral fusion category there is another useful grading called the \textit{GN-grading} first studied in \cite{GN2}:
  \begin{theorem}\cite[Theorem 3.10]{GN2}
    Let $\mcC$ be a weakly integral fusion category. Then there is an elementary
    abelian $2$-group $E$, a set of distinct square-free positive
    integers $n_x$, $x \in E$, with $n_0 = 1$, and a faithful grading $\mcC =
    \oplus_{x\in E} \mcC(n_x)$ such that $\dim(X) \in \mathbb Z \sqrt{n_x}$
    for each $X \in \mcC(n_x)$.
  \end{theorem}
Notice the trivial component of the $GN$-grading is the integral subcategory $\mcC_{\intcat}$.

Our approach to classification relies upon \textit{equivariantization} and its inverse functor 
  \textit{de}-\textit{equivariantization} (see
  for example \cite{DGNO1}) which we now briefly describe.
  An \textit{action} of a finite group $G$ on a fusion category $\mcC$ is a
  strong tensor functor $\rho: \underline{G}\to \End_{\otimes}(\mcC)$.  The
  \textit{$G$-equivariantization} of the category $\mcC$ is the category $\mcC^G$
  of $G$-equivariant objects and morphisms of $\mcC$. When $\mcC$ is a fusion
  category over an algebraically closed field $\mbbK$ of characteristic $0$, the
  $G$-equivariantization, $\mcC^G$, is a fusion category with  $\dim \mcC^G = |G| \dim \mcC$. The fusion
  rules of $\mcC^G$ can be determined in terms of the fusion rules of the
  original category $\mcC$ and group-theoretical data associated to the group
  action \cite{BuN1}. De-equivariantization is the inverse to equivariantization. Given a fusion category $\mcC$ and a
  Tannakian subcategory $\Rep G\subset\mcC$, consider the algebra $A = \Fun\paren{G}$ of
  functions on $G$. Then $A$ is a commutative algebra in $\mcC$.
  The category of $A$-modules on $\mcC$, $\mcC_{G}$, is a
  fusion category called a \textit{$G$-de-equivariantization} of
  $\mcC$.  We have
  $\abs{G}\dim \mcC_G = \dim \mcC$, and there are
  canonical equivalences $\paren{\mcC_G}^G \cong \mcC$ and $\paren{\mcC^G}_G
  \cong \mcC$. An important property of the de-equivariantization is that if the
  Tannakian category in question is $\mcC'$, then $\mcC_{G}$ is modular and is
 called the \textit{modularization} of $\mcC$ \cite{Brug1,M5}.

  One may also construct modular categories from a given modular
  category $\mcC$ with an action of a finite group
  $G$, by the \textit{gauging} introduced and studied in \cite{BBCW1} and \cite{CGPW1}.  Gauging is a $2$-step
  process. The first step is to extend the modular category $\mcC$ to a
  $G$-crossed braided fusion category $\mcD\cong\bigoplus_{g\in G}\mcD_g$ such that $\mcD_e=\mcC$ with a $G$-braiding, see \cite{DGNO1}. The second step is to
  equivariantize $\mcD$ by $G$.  One important remark is that there are certain obstructions
  to the first step \cite{ENO3}, while the second is always possible.

\section{Even metaplectic categories of dimension $8N$, with $N$ an odd integer}
  \label{Section: Metaplectic Classification}

  In this section we consider metaplectic modular categories. A general
  understanding of these categories will facilitate a classification of modular
  categories of dimension $8N$ for any odd square-free integer $N$. Metaplectic
  categories were defined and studied in \cite{ACRW1}. For convenience we recall
  the definition here.
  \begin{defn}
    An \textbf{even metaplectic modular category} is a modular category
    $\mcC$ with positive dimensions that is Grothendieck equivalent to $SO(2N)_2$, for some integer $N
    \geq 1$.
  \end{defn}

  \begin{rmk}
    There are significant differences between $N$ odd and even, see \cite{NR1}.
    For our results only the case $N$ odd plays a role.  Notice the case $N=1$
    is degenerate, corresponding to $SO(2)_2$ which has fusion rules like
    $\mbbZ_8$.  Our results carry over to this (pointed) case with little
    effort, so we include it.
  \end{rmk}

  Throughout the remainder of this section, $\mcC$ will denote an even
  metaplectic category with $N$ an odd integer. In particular, $\mcC$ has
  dimension $8N$.  In this case it follows from the definition that $\mcC$ has
  rank $N+7$ \cite{NR1}. While the dimension and rank of $\mcC$ are
  straightforward, the fusion rules strongly depend on the parity of $N$. While
  the full fusion rules and $S$-matrix can be found in \cite[Subsection
  3.2]{NR1}, we review here some of the fusion rules that will be relevant in
  later proofs. The group of isomorphism classes of invertible objects of
  $\mcC$ is isomorphic to $\mbbZ_4$. We will denote by $g$ a generator of this
  group, and will abuse notation referring to the invertible objects as $g^{k}$.
  The only non-trivial self-dual invertible object is $g^2$. In $\mcC$, there
  are $N-1$ self-dual simple objects, $X_i$ and $Y_{i}$, of dimension $2$. Furthermore, one may order the $2$-dimensional objects such that $X_{1}$ generates
  $\mcC_{\ad}$ and $Y_{1}$ generates $\mcC_{\intcat}$.  The remaining four
  simples in $\mcC$, $V_{i}$, have dimension $\sqrt{N}$.  $\mcC$ has the following fusion rules:
  \begin{itemize}
    \item $g\otimes X_a\simeq Y_{\frac{N+1}{2}-a}$, and $g^2\otimes X_a\simeq
    X_a,$ and $g^{2}\otimes Y_{a}\simeq Y_{a}$ for  $1\leq a\leq
    \paren{N-1}/2$.
    \item $X_a\otimes X_a = \1\oplus g^2\oplus X_{\mathrm{min}\{2a, N-2a\}}$; $X_{a}\otimes X_{b}=X_{\mathrm{min}\{a+b,N-a-b\}}\oplus X_{\abs{a-b}}$  ($a\neq b$)
    
    \item $V_1\otimes V_1 = g\oplus
    \bigoplus\limits_{a= 1}^{\frac{N-1}{2}}Y_{a}$.
    \item $gV_{1}=V_{3}$, $gV_{3}=V_{4}$, $gV_{2}=V_{1}$, $gV_{4}=V_{2}$ and $g^{3}V_{a}=V_{a}^{*}$, $V_{2}=V_{1}^{*}$, $V_{4}=V_{3}^{*}$
  \end{itemize}

  We remark that it is immediately clear that an even metaplectic modular
  category of dimension $8N$ with $N$ odd is prime (not the Deligne product of
  two modular subcategories): $V_1$ is a tensor generator, hence cannot reside
  in any proper fusion subcategory.

  For $\SO\paren{2N}_{2}$ we order the simple objects by their hightest weights
  as follows:\ \\
  $\mathbf{0},2\l_{1},2\l_{N-1},2\l_{N},\l_{1},\ldots,\l_{N-2},\l_{N-1}+\l_{N},
  \l_{N-1},\l_{N},\l_{N-1}+\l_{1},\l_{N}+\l_{1}$
  The correspondence with the notation used in this paper is
  $\1,g^{2},g,g^{3},Y_{1},X_{1},\ldots,Y_{\frac{N-1}{2}},X_{\frac{N-1}{2}},V_{1},V_{2},V_{3},V_{4}$. The $S$- and $T$-matrices have the following forms:
  \begin{align*}
    S&=\paren{\begin{smallmatrix}
      A & B & C  \\
      B^t & D & 0 \\
      C^t & 0 & E
      \end{smallmatrix}},
    \qquad
    T=\diag\paren{1,1,i^{N},i^{N},\ldots,(-1)^{a}e^{-a^2\pi
    i/2N},\ldots,\theta_{\epsilon},\theta_{\epsilon},-\theta_{\epsilon},-\theta_{\epsilon}}
  \end{align*}
  where $\theta_{\epsilon}=e^{\pi i\frac{2N-1}{8}}$, and $A$, $E$, and $C$ are $4\times 4$ matrices, $B$ is a
  $4\times\paren{N-1}$ matrix and $D$ is $\paren{N-1}\times\paren{N-1}$.
  To give their explicit forms we first define a $\mbbC$-valued function:
  \begin{equation*}
    G(N) = \begin{cases}
       (\frac{-1}{N})\sqrt{N} & \text{ if } N\equiv 1 \mod{4}\\
       i(\frac{-1}{N})\sqrt{N} & \text{ if } N\equiv 3 \mod{4}\\
    \end{cases}
  \end{equation*}
  where $\paren{\frac{-1}{N}}$ is the Jacobi symbol and set $\alpha = \bar{G\paren{N}}\th_{\varepsilon}^{2}$  and
  $\beta = i^{2N-1}\th_{\varepsilon}^{2} G\paren{N}$. Then:
  \begin{align*}
    A &= \paren{\begin{smallmatrix}
      1 & 1 & 1 & 1  \\
      1 & 1 & 1 & 1  \\
      1 & 1 & -1 & -1  \\
      1 & 1 & -1 & -1
    \end{smallmatrix}},
    \quad
    B = 2\paren{\begin{smallmatrix}1&\cdots&1\\1&\cdots&
    1\\-1\cdots&(-1)^a&\cdots 1\\-1 \cdots&(-1)^a&\cdots 1\end{smallmatrix}},
    \quad C = \sqrt{N}\paren{\begin{smallmatrix}
      1 & 1 & 1 & 1  \\
      -1 & -1 & -1 & -1  \\
      -i^N & i^N & -i^N & i^N  \\
      i^N & -i^N & i^N & -i^N
    \end{smallmatrix}},\\
    E&=\paren{\begin{smallmatrix}
      %
      %
      \bar{\alpha} & \alpha & \beta & \bar{\beta}\\
      \alpha & \bar{\alpha} & \bar{\beta} & \beta\\
      \beta & \bar{\beta} & \bar{\alpha} & \alpha\\
      \bar{\beta} & \beta & \alpha & \bar{\alpha}
    \end{smallmatrix}},
    \quad
    D_{a,b} = 4 \cos(\frac{\pi ab}{N}) \quad
   1\leq a,b\leq N-1.
  \end{align*}

  Finally, we note that the in terms of simple objects the basis in which $S$ and $T$ is expressed is ordered such that the first two
  columns of $A$ correspond to invertibles in $\mcC_{\ad}$, and the even indexed
  columns of $B$ correspond to the $2$-dimensional objects in $\mcC_{\ad}$. 

  \begin{lemma}
    \label{boson}
    If $\mcC$ is an even metaplectic modular category, then the unique non-trivial self-dual invertible object
    is a boson.
  \end{lemma}
  \begin{proof}
    For $N=1$, we have the degenerate case where $\mcC$ is Grothendieck
    equivalent to $\Rep\paren{\mbbZ_{8}}$ with twists $\th_{j}=e^{j^2\pi i
    /8}$. There is a unique non-trivial self-dual object, $j=4$, and the
    corresponding twist is $\th_{4} = e^{16\pi i /8}=1$, thus
    it is a boson.

    If $N \geq 2$, consider $X_i$ one of the $2$-dimensional simple objects in
    $\mcC_{\ad}$ and recall that $g^2\otimes X_i = X_i$. Of course,
    $g^{2}\in\mcC_{\ad}$. Since $\mcC$ is
    modular we know that $C_{\mcC}\(\mcC_{\ad}\)=\mcC_{\pt}$ \cite{GN2}. In particular,
    $g^{2}$ is in $C_{\mcC_{\ad}}\(\mcC_{\ad}\)$. Thus $g^{2}$ is not a fermion
    by \cite[Lemma 5.4]{M5}.
  %
  \end{proof}

  It follows such a category, $\mcC$, contains a Tannakian category
  equivalent to $\Rep(\mbbZ_{2})$.
  \begin{defn}
    A \textit{cyclic modular category} is a modular category that is
    Grothendieck equivalent to $\Rep(\mbbZ_n)$ for some integer $n$. When the
    specific value of $n$ is important we will refer to such a category as a
    $\mbbZ_n$-cyclic modular category.
  \end{defn}

  For more details regarding cyclic modular categories see \cite[Section
  2]{ACRW1}.
  \begin{lemma}
    If $\mcC$ is an even metaplectic modular category then $\mcC_{\mbbZ_2}$ is a  generalized
    Tambara-Yamagami category of dimension $4N$. In particular, the trivial
    component of $\mcC_{\mbbZ_2}$ is a cyclic modular category of dimension
    $2N$.
  \end{lemma}
  \begin{proof}
    As we noted previously, by \lemmaref{boson}, $\langle
    g^{2}\rangle\cong\Rep\(\mbbZ_2\)$ is a Tannakian subcategory of $\mcC$. In
    particular we can form the de-equivariantization $\mcC_{\mbbZ_2}$ which is a
    braided $\mbbZ_2$-crossed fusion category of dimension $4N$. In particular,
    the trivial component of $\mcC_{\mbbZ_2}$ under the $\mbbZ_2$-crossed
    braiding is modular of dimension $2N$ \cite[Proposition 4.56(ii)]{DGNO1}.

    Since $g^{2}$ fixes the 2-dimensional objects of $\mcC$, these
    2-dimensionals give rise to pairs of distinct invertibles in
    $\mcC_{\mbbZ_2}$. This produces $2(N-1)$ invertibles in $\mcC_{\mbbZ_2}$. On
    the other hand, $g^2\otimes 1=g^{2}$ and $g^{2}\otimes g=g^{3}$. This leads
    to two more invertibles in $\mcC_{\mbbZ_2}$. Since the non-integral objects
    are moved by $g^{2}$, we can conclude that $\mcC_{\mbbZ_2}$ has only two
    non-integral objects which each have dimension $\sqrt{N}$. In particular,
    $\mcC_{\mbbZ_2}$ is generalized Tambara-Yamagami category and the the
    trivial component of $\mcC_{\mbbZ_2}$ is a pointed modular category with
    $2N$ invertible objects \cite{Lip1}. To conclude that this pointed category
    is cyclic we apply the proof of \cite[Lemma 3.4]{ACRW1} \textit{mutatis
    mutandis}.
  \end{proof}

  \begin{remark}\label{gauging}
    By the previous lemmas, every even metaplectic category with $N$ odd, can
    obtained as a $\mbbZ_2$-gauging of a cyclic modular category of dimension
    $2N$. We refer the reader to \cite{BBCW1} and \cite{CGPW1} for a precise
    definition of gauging and its properties.
  \end{remark}

  So to classify even metaplectic categories with $N$ odd, we must understand
  $\mbbZ_2$ actions by braided tensor autoequivalences of a cyclic modular category of
  dimension $2N$.  Note that a $\mbbZ_{n}$-cyclic modular category with $n$ odd
  has a fixed-point-free $\mbbZ_2$ action by braided tensor autoequivalences associated to
  the $\mbbZ_{n}$ group automorphism $j\mapsto n-j$. We refer to this
  automorphism as the \textit{particle-hole symmetry}.  In fact, it is shown in
  \cite{ACRW1} that for $n=p^a$ an odd prime power, this is the only non-trivial
  $\mbbZ_2$ action by braided tensor autoequivalences.  By decomposing a $\mbbZ_n$-cyclic modular
  category into its prime power Deligne tensor factors, we see that every
  $\mbbZ_2$ action by braided tensor autoequivalences on a $\mbbZ_{n}$ cyclic modular
  category (for $n$ odd) is obtained by choosing either particle-hole symmetry or
  the identity on each $\mbbZ_{p^a}$ cyclic modular Deligne factor.  Of course
  particle-hole symmetry for $\mbbZ_n$ corresponds to chosing particle-hole
  symmetry on each Deligne factor. On the other hand, the Semion (or its
  conjugate) has exactly one non-trivial $\mbbZ_2$ action by braided tensor
  autoequivalences corresponding to the non-trivial element of
  $H^2(\mbbZ_2,\mbbZ_2)=\mbbZ_2$ \cite{BBCW1}.

  For a $\mbbZ_{2N}$ cyclic modular category ($N$ odd) let us define the
  particle-hole symmetry to be the $\mbbZ_2$ action by braided tensor autoequivalences
  corresponding to the non-trivial tensor autoequivalence on the Semion factor
  and ordinary particle-hole symmetry on the odd $\mbbZ_N$ factor.

  \begin{remark}
    \label{nontriv}
    Let $\mathcal{C}$ be a $\mbbZ_{2N}$-cyclic modular category.  It follows from
    \cite{BBCW1,CGPW1} that if we choose a $\mbbZ_2$ action by braided tensor
    autoequivalences on $\mathcal{C}$ that restricts to the identity on some
    $\mbbZ_{p^a}$-cyclic modular subcategory $\mathcal{P}_{p^a}$ of $\mathcal{C}$
    (including $p=2$, $a=1$) then the corresponding gauging $\mathcal{C}^{\times,
    G}_G$ will have a Delinge factor of $\mathcal{P}_{p^a}$.  Indeed, in the
    extreme case of $N=1$, if we gauge $Sem$ by the trivial $\mbbZ_2$ action we
    obtain $Sem\boxtimes \Rep(D^\omega \mbbZ_2)$ for some cocycle $\omega$.
    Indeed, gauging $\Vec$ by the trivial $\mbbZ_2$ action produces the factor on the
    right.
  \end{remark}

  \begin{theorem}
    \label{Theorem: Metaplectic main result}
    If $\mcC$ is an even metaplectic category of dimension $8N$, with $N$ an odd
    integer, then $\mcC$ is a gauging of the particle-hole symmetry of a
    $\mbbZ_{2N}$-cyclic modular category. Moreover, for $2N = 2p_1^{k_1} \cdots
    p_r^{k_{r-1}}$, with $p_i$ distinct odd primes, there are exactly $2^{r+2}$
    many inequivalent even metaplectic modular categories.
  \end{theorem}
  \begin{proof}
    By \rmkref{gauging}, each even metaplectic modular category, with $N$ odd,
    is obtained as a $\mathbb Z_2$-gauging of a $\mbbZ_{2N}$-cyclic modular
    category.  However, since an even metaplectic modular category is prime,
    \rmkref{nontriv} implies that it can only be obtained by gauging
    particle-hole symmetry. This proves the first statement.

    To count the inequivalent even metaplectic categories of dimension $8N$ with
    $N$ odd, we note that for each of the $r+1$ prime divisors of $2N=2p_1^{k_1} \cdots p_r^{k_r}$ there are
    exactly two cyclic modular categories (see \cite[Section 2]{ACRW1}). Gauging the
    particle-hole symmetry leads to an additional choice of $H^3(\mathbb Z_2,
    U(1))=\mathbb Z_2$ yielding a total of $2^{r+2}$ choices.
  \end{proof}

\section{Modular categories of dimension $8m$, with $m$ odd square-free integer}
  Metaplectic categories are a large class of dimension $8m$ weakly integral
  modular categories for $m$ odd and square-free.  In
  this section we show that the only
  non-metaplectic modular categories of dimension $8m$ are pointed or
  products of smaller categories. In order to accomplish this
  we will first determine the structure of the pointed subcategory. This will
  enable us to show that the
  integral subcategory is Grothendieck equivalent to $\Rep\paren{D_{4m}}$. This
  will be sufficient to establish that $\mcC$ is metaplectic if it is prime.
  Along the way we will resolve the case that $\mcC$ has dimension $p^{3}m$ for
  $p$ an odd prime, and $m$ a square-free integer.

  We begin by considering integral categories of dimension $p^{2}m$ and
  $p^{3}m$. This will allow us to quickly reduce to the case of $8m$.
  \begin{lemma}
    \label{Lemma: pkm}
    Let $p$ be a prime and $m$ a square-free integer such that $\gcd(m,p)=1$. If
    $\mcC$ is a modular category of dimension $p^{k}m$ with
    $p^{k}\mid\dim \mcC_{\pt}$, then $\mcC$ is pointed.
  \end{lemma}
  \begin{proof}
    First note that $\dim\mcC_{\ad}=\frac{p^{k}m}{\dim\mcC_{\pt}}$. Since
    $p^{k}\mid\dim\mcC_{\pt}$ we know that $\dim\mcC_{\ad}\mid m$. On the other
    hand, $\mcC_{\ad}\subset\mcC_{\intcat}$ and all simples in $\mcC_{\intcat}$
    have dimension $p^j$, $j\geq 0$. Thus the dimensions of simples in $\mcC_{\ad}$
    are $1$ since $\gcd(p,m)=1$ and $p\ndiv\dim\mcC_{\ad}$, thus
    $\mcC_{\ad}\subset\mcC_{\pt}$. In particular, $\mcC$ is nilpotent and is
    given by
    $\mcC_{p^{k}}\boxtimes\mcC_{q_{1}}\boxtimes\cdots\boxtimes\mcC_{q_{\ell}}$
    where $q_{j}$ are primes such that $m=q_{1}q_{2}\cdots q_{\ell}$ and
    $\mcC_{k}$ is a modular category of dimension $k$, see \cite[Theorem
    1.1]{DGNO2}. By the classification of prime dimension modular categories we
    know $\mcC_{q_{1}}\boxtimes\cdots\boxtimes\mcC_{q_{\ell}}$ is pointed,  but
    $p^{k}$ does not divide its dimension. Thus $\mcC_{p^{k}}$ must be pointed.
  \end{proof}

  For an integral modular category of dimension $p^{2}m$ with $p$ and $m$ as in
  the above lemma, we have that $p^{2}m= \dim \mcC = \abs{\mcC_{\pt}}+ap^{2}$,
  where $a$ is the number of simple objects of dimension $p$. Thus we have the
  following corollary.
  \begin{cor}
    \label{Cor: p2m integral}
    If $p$ is prime and $m$ is a square-free integer with $\gcd\(m,p\)=1$, and
    $\mcC$ is an integral modular category with dimension $p^{2}m$, then $\mcC$
    is pointed.
  \end{cor}

  \begin{lemma}
    \label{Lemma: integral p3m}
    Any integral modular category $\mcC$ of dimension $p^{3}m$ where $p$ is
    a prime and $m$ is a square-free integer is
    pointed.
  \end{lemma}
  \begin{proof}
  The case $m=1$ (i.e. $p^3$) is well-known and easily verified using the grading and $|\mcU(\mcC)|=\rank\mcC_{\pt}$.
  
  Assume that $\mcC$ is not pointed.
 First we consider $\dim\mcC=p^4$.  By \cite[Lemma 4.9]{DN} $\mcC$ is a Drinfeld center $\mcZ(\Vec_G^\omega)\cong \Rep(D^\omega G)$ for some group $G$ of order $p^2$.  By \cite{Ng03} such a $D^\omega G$ is commutative, i.e. $\mcC$ is pointed.
  
    Since $\mcC$ is modular we know that the simples have dimensions $1$ and/or
    $p$. In particular, there is an integer $b$ such that
    $p^{3}m=\dim\mcC_{\pt}+p^{2}b$ and thus $p^{2}\mid \dim\mcC_{\pt}$. By
    \lemmaref{Lemma: pkm} we may assume $p^{3}$ does not divide $\dim\mcC_{\pt}$.
    In particular, there exists an integer $k$ such that $k\mid m$ and
    $\dim\mcC_{\pt}=p^{2}k$. Next let $\mcC_{i}$ denote the components of the
    universal grading of $\mcC$.  Then $pm/k=\dim\mcC_{i}=a_{i}+b_{i}p^{2}$ where
    $a_{i}$ is the number of invertibles in component $\mcC_{i}$ and $b_{i}$ is
    the number of dimension $p$ objects in $\mcC_{i}$. So we can conclude that
    $p\mid a_{i}$ and $a_{i}\neq 0$. Thus $\dim\mcC_{\pt}=\sum_{i}a_{i}\geq
    \sum_{i}p=p\dim\mcC_{\pt}$, an impossibility.
  \end{proof}

  Next we note that the condition that $\mcC$ is integral in the previous
  statements is vacuous if $p$ is odd. This is made explicit by the
  following lemma.
  \begin{lemma}
    \label{Lemma: 4 divides}
    If $\mcC$ is a strictly weakly integral modular category, then
    $4\mid\dim\mcC$.
  \end{lemma}
\begin{proof}
   Coupling the weakly integral grading of \cite{GN2} and the fact that
   $\mcC$ is strictly weakly integral, we have $2\mid
   \dim\mcC_{\pt}\mid\dim\mcC_{\intcat}\mid\dim\mcC/2$.  \end{proof}

  \lemmaref{Lemma: 4 divides}, \corref{Cor: p2m integral}, and \cite[Theorem
  3.1]{BGNPRW1} resolve the case of $\mcC$ a weakly integral modular
  category of dimension $p^{2}m$ with $m$ square-free and coprime to $p$. We
  have:
  \begin{theorem}
    \label{Theorem: p2m main theorem}
    If $p$ is a prime, $m$ is a square-free integer coprime to
    $p$, and $\mcC$ is non-pointed modular category of dimension $p^{2}m$,
    then $p=2$ and one of the following is true:
    \begin{itemize}
      \item[(i)]  $\mcC$ contains an object of dimension $\sqrt{2}$ and is
      equivalent to a Deligne product of an Ising modular category with a cyclic
      modular category or
      \item[(ii)] $\mcC$ contains no objects of dimension $\sqrt{2}$ and is
      equivalent to a Deligne product of a $\mbbZ_{2}$-equivariantization of a
      Tambara-Yamagami category over $\mbbZ_{k}$ and a cyclic modular category
      of dimension $n$ where $1\leq n=m/k\in\mbbZ$.
    \end{itemize}
  \end{theorem}
Applying
  Lemmas \ref{Lemma: pkm} and \ref{Lemma: integral p3m} we have:
  \begin{cor}
    If $\mcC$ is a modular category of dimension $8m$ with $m$ an odd square-free
    integer, then either $\mcC$ is pointed or $\mcC$ is strictly weakly integral
    and $8\ndiv\abs{\mcU\(\mcC\)}$.
  \end{cor} 
  
  So to study weakly integral modular categories of dimension $p^{3}m$ it
  suffices to study strictly weakly integral modular categories of dimension
  $8m$ by Lemmas \ref{Lemma: 4 divides} and \ref{Lemma: integral p3m}.

  \begin{lemma}
    \label{Lemma: Equivariantizations contain Ising}
    Let $G$ be a finite group and $\mcC$ be the $G$-equivariantization of a fusion category $\mcD$ which contains an Ising category. If $\mcC$ has the
    property that every
    $X\in\Irr\(\mcC\)$ such that $\dim X\in\mbbZ\[\sqrt{2}\]$ actually has
    dimension $\sqrt{2}$, then $\mcC$ contains an Ising category.
  \end{lemma}
  \begin{proof}
    By \cite{BuN1}, the objects in $\mcC$ are pairs $S_{X,\p}$ where $X$ is
    the $G$-orbit of a simple in $\mcD$ and $\p$ is a projective
    representation of $\Stab_{G}\(X\)$. Now let $\s$ be the Ising object in
    $\mcD$. Then there exists a projective representation of $\Stab_{G}\(\s\)$
    that is self-dual, denote it by $\r$. Then $S_{\s.\r}$ is a simple in
    $\mcC$ and is self-dual by \cite{BuN1}. Moreover, $\dim
    S_{\s,\r}=\sqrt{2}\dim\r\[G:\Stab_{G}\(X\)\]$. However, our hypotheses
    regarding $\mcC$ ensure that $\dim S_{\s,\r}=\sqrt{2}$.
  \end{proof}
  The next two lemmas deal with the cases $\dim\mcC=8$ or $16$, i.e. $m=1,2$.
  \begin{lemma}
    \label{Cor: Dimension 8 SWI Fusion Contains Ising}
    If $\mcC$ is a dimension $8$ strictly weakly integral premodular category,
    then $\mcC\cong\mcI\boxtimes\mcD$ where $\mcD$ is a $\mathbb Z_2$-cyclic premodular category.
  \end{lemma}
  \begin{proof}
     Under the GN-grading (either $(\mbbZ/2\mbbZ)^2$ or $\mbbZ/2\mbbZ$), $\mcC_0=\mcC_{\intcat}$ so $\dim\mcC_{\intcat}=2$ or $4$ (respectively). The first case would yield 3 objects of dimension $\sqrt{2}$, which is inconsistent with the GN-grading. Thus $\dim\mcC_{\intcat}=4$ and we conclude that $\mcC_{\intcat}=\mcC_{\pt}$ and there are two
distinct objects $X$ and $Y$ of dimension $\sqrt{2}$.

Observe that if \textit{either} the universal grading group $\mcU\(\mcC\)\cong(\mbbZ/2\mbbZ)^2$ or $X\cong X^*$ then $X^{\otimes 2}=\1\oplus z\in\mcC_{\ad}$ with $z\not\cong\1$ and $z^*\cong z$.  Thus $X,\1,z$ form a braided Ising category $\mcI$ by \cite{DGNO1}.  Now it follows from a dimension count and \cite[Cor. 7.8]{M2} that $\mcC\cong\mcI\boxtimes\mcD$ as claimed.

Thus it is enough to consider the case that $X^*\cong Y$.  Then we have $X\otimes Y\cong \1\oplus z$ for some self-dual $z\in\mcC_{\ad}$ with $\theta_z^4=1$ (i.e. $z$ is a boson, a fermion or a semion).  If $\theta_z=1$ then $\langle z\rangle$ is Tannakian so that we may de-equivariantize $\mcC$.  However, since $(\1\oplus z)\otimes X\cong 2X$, under the de-equivariantization functor $F$ we have $F(X)=X_1\oplus X_2$ with $\dim(X_i)=\frac{\sqrt{2}}{2}$ (see \cite[Prop. 2.15]{M3}) which is not an algebraic integer.  So $\theta_z\neq 1$.  Now if the remaining two invertible objects satisfy $\theta_a=\theta_b$, the balancing equation implies  $S_{z,b}=S_{z,a}=\frac{1}{\theta_z}\neq 1$  hence the M\"uger centers of both $\mcC$ and $\mcC_{\intcat}$ are trivial.  But by \cite{M2} we would then have $\rank{\mcC_{\intcat}}=4\mid \rank{\mcC}=6$, a contradiction.  Thus we must have $\theta_a\neq\theta_b$ (so $a\ncong b^*$) and $\theta_z\neq 1$.  Now if either $a$ or $b$ is a boson we may de-equivariantize and then apply \lemmaref{Lemma: Equivariantizations contain Ising} to conclude that $X$ generates an Ising category, contradicting $X\ncong X^*$.  So each of $z,a$ and $b$ are fermions or semions. Now applying balancing and the fact that $a\otimes X\cong X^*\cong b\otimes X$ we have $S_{z,X}=\frac{d_X}{\theta_z}\neq d_Xd_z$, $S_{a,X}=\frac{d_X}{\theta_a}\neq d_Xd_a$ and $S_{b,X}=\frac{d_X}{\theta_b}\neq d_Xd_b$ so that $\mcC$ is modular.  Now if any of $a,b$ or $z$ is a semion we may factor $\mcC$ as a Delinge product $\Sem\boxtimes \mcE$ again contradicting $X\ncong X^*$.  The only remaining possibility is that each of $a,b$ and $z$ is a fermion.  From the balancing equation we compute the $S$-matrix of $\mcC_{\intcat}$ and find that $S_{i,j}=-1$ for $\1\neq i\neq j\neq \1$ which implies $\mcC_{\intcat}$ is modular and we obtain the contradiction $4\mid 6$ as above.

  \end{proof}

  \begin{lemma}
    \label{Lemma: Dimension 16 SWI MTC Contains Ising}
    If $\mcC$ is a strictly weakly integral modular category of dimension $16$,
    then $\mcC$ contains an Ising subcategory. In particular,
    $\mcC\cong\mcI\boxtimes\mcD$ where $\mcD$ is Ising or a pointed modular
    category of dimension $4$.
  \end{lemma}
  \begin{proof}
 The squares of the dimensions of the simple objects must divide $16$ and so the possible simple object dimensions are $1$, $2$, $\sqrt{2}$, and $2\sqrt{2}$. In particular, the GN grading is $\mbbZ/2\mbbZ$ and the dimension of the integral component is $8$. 
 
 There are now five possibilities for the universal grading group:
    $\mbbZ/2\mbbZ$, $\mbbZ/4\mbbZ$, $\(\mbbZ/2\mbbZ\)^{2}$, $\(\mbbZ/2\mbbZ\)^{3}$, or $\mbbZ/8\mbbZ$.
    The first case is not possible by dimension count in the integral
    component. 
    
    In the case that the universal grading has order $8$, each component has dimension $2$. Then the integral components are rank $2$ pointed categories and the non-integral components each have a single simple object of dimension $\sqrt{2}$. Then the category is a modular generalized Tambara-Yamagami category and $\mcC\cong\mcI\boxtimes\mcD$ where $\mcD$ is a pointed modular category of dimension $4$ by \cite[Theorem 5.4]{Nat}.
    
    In the case that the universal grading has
    order $4$, dimension count reveals that there are four invertible objects, one object of dimension $2$, and four objects of dimension $\sqrt{2}$. Moreover, $\mcC_{\ad} = \mcC_{\pt}$. Since there is only one simple object of dimension $2$, it is self-dual. Moreover, $\mcC_{\intcat}$ is a braided Tambara-Yamagami category, and therefore $G(\mcC)\cong U(\mcC) \cong \mathbb Z_2\times \mathbb Z_2$, by \cite[Theorem 1.2]{Si1}. Then the fusion subcategory generated by the adjoint component and one non-integral component has dimension $8$. Then, by \corref{Cor: Dimension 8 SWI Fusion Contains Ising}, it contains Ising and the result follows.
    


  \end{proof}

  \begin{prop}
    \label{prop: prime or classified}
    If $\mcC$ is a strictly weakly integral modular category of dimension $8m$
    where $m$ is an odd square-free integer, then one of the following is true:
    \begin{itemize}
      \item[(i)] $\mcC$ is a Deligne product of
      a prime modular category
      of dimension $8\ell$ and a pointed $\mbbZ_k$-cyclic modular
      category with $k$ odd, or
      \item[(ii)] $\mcC=\Sem\boxtimes\mcD$, where $\Sem$ is a Semion category
      and $\mcD$ is a strictly weakly integral modular category of dimension
      $4m$ (see \cite{BGNPRW1}).
    \end{itemize}
  \end{prop}
  \begin{proof}
    If $\mcC$ is prime then we are in case (i) with $k=1$.  If $\mcC$ is not
    prime, then it must factor into a Deligne product of modular categories
    \cite[Theorem 4.2]{M2}. By \lemmaref{Lemma: 4 divides} we can conclude that
    $\mcC=\mcC_{1}\boxtimes\mcC_{2}$, $\mcC_{1}$ is strictly weakly integral,
    and $\mcC_{2}$ is cyclic, pointed and non-trivial. The result now follows
    from earlier work \cite{RSW,BGNPRW1}: if $\dim\mcC_2$ is odd
    then we are in case (i), otherwise $2\mid \dim\mcC_2$, and hence $\mcC$
    contains a Semion category \cite{ACRW1}.
  \end{proof}

 We have now reduced to the case that $\mcC$ is a prime
  strictly weakly integral modular category of dimension $8m$ with $m$ and odd
  square-free integer, and we assume $\mcC$ has this form in what follows.
  \begin{prop}
    \label{semion}
    Suppose $\mcD$ is a non-symmetric premodular category that is Grothendieck
    equivalent to $\Rep\(\mbbZ_{2}\times\mbbZ_{2}\)$.  If  $\mcD$ contains the
    symmetric category $\Rep \(\mbbZ_2\)$, then $\mcD$ contains a Semion.
  \end{prop}
  \begin{proof}
    Denote the simple objects in $\mcD$ by $1,g,h,gh$ and $g$ generates the
    symmetric category $\Rep \(\mbbZ_2\)$. Then examining the balancing relations
    for $S_{h,h},S_{g,h},S_{gh,gh}$,
    and $S_{h,gh}$ we see that $\th_{h}=\th_{gh}$,
    $S_{h,h}=S_{gh,gh}=S_{h,gh}=\th_{h}^{2}$. Since $\mcD$ is not symmetric, but
    $\langle g \rangle$ is, we can conclude that $S_{h,h}\neq 1$. However,
    $\mcD$ is self-dual and hence $S_{h,h}$ is a real root of unity. In
    particular, $S_{h,h}=-1$ and $\th_{h}=\pm i$.  Consequently, $h$ generates a
    Semion subcategory.
  \end{proof}

  \begin{corollary}\label{Z_4}
    $\mcU\(\mcC\)\cong\mbbZ/4\mbbZ$.
  \end{corollary}
  \begin{proof}
    By \lemmaref{Lemma: pkm} we know that $8$ does not divide $\dim\mcC_{\pt}$.
    Moreover, by invoking the universal grading and examining the dimension
    equation for $\mcC_{\ad}$ modulo 4 we see that there exists an odd integer
    $k$ such that $\dim\mcC_{\pt}=4k$. If $k>1$, there exists an
    odd prime $p$ dividing $k$ and a $p$-dimensional category
    $\mcD\subset\mcC_{\pt}$.

    Since $\mcC$ is prime we know that $\mcD$ is symmetric and Tannakian.
    De-equivariantizing $\mcC$ by $\mbbZ_{p}$, the trivial component of the
    resulting $\mbbZ_{p}$-graded category has dimension $8m/p^{2}$
    \cite[Proposition 4.56(i)]{DGNO1}. This is not possible as $p$ is odd and $m$
    is square-free. Thus $\dim\mcC_{\pt}=4$.

    By dimension count we can deduce that $\dim\(\mcC_{\ad}\)_{\pt}=2$. Since
    $C_{\mcC}\mcC_{\pt}=\mcC_{\ad}$ we can deduce that $\mcC_{\pt}$ is not symmetric.
    However, $\mcC$ is prime and so $\Vec\neq
    C_{\mcC_{\ad}}\(\mcC_{\ad}\)\subset\(\mcC_{\ad}\)_{\pt}$ and so
    $C_{\mcC_{\ad}}\(\mcC_{\ad}\)=\(\mcC_{\ad}\)_{\pt}$. Letting $g$ be a
    generator of this M\"{u}ger center and $X$ a 2-dimensional object in
    $\mcC_{\ad}$, then by dimension count $g$ is a subobject of $X\otimes
    X^{*}$. In particular, $g$ fixes $X$ and thus $\th_{g}=1$ by \cite[Lemma
    5.4]{M5}.

    By applying Proposition \ref{semion} to $\mcD = \mcC_{\pt}$ and invoking the
    primality of $\mcC$ we have that $\mcC_{\pt}\ncong \Rep\(\mathbb
    Z_2\times\mathbb Z_2\)$.
  \end{proof}

  Henceforth, let $g$ be a generator of $\mcC_{\pt}$, and $\mcC_{g^{k}}$ the
  component of the universal grading of $\mcC$ corresponding to the simple
  $g^{k}\in\mcC_{\pt}$. Furthermore note that, $\mcC_{\1}$ and
  $\mcC_{g^{2}}$ each contain exactly two invertible objects. Finally, we will
  denote by $X_{1},X_{2},\ldots, X_{n}$ the simple 2-dimensional objects in
  $\mcC_{\ad}$ and $Y_{1}=g\otimes X_{i}$ the simple 2-dimensional objects
  in $\mcC_{g^{2}}$.
  \begin{rmk}
    \label{rmk: Tannakian}
    Under this notation $\(\mcC_{\ad}\)_{\pt}=\langle g^{2}\rangle$ is a
    2-dimensional Tannakian category.
  \end{rmk}

  \begin{lemma}
    \label{Lemma: self-dual}
    All 2-dimensional simple objects in $\mcC$ are self-dual.
  \end{lemma}
  \begin{proof}
    Recall that under our notation, $X_{i}$ are the simple 2-dimensional objects
    in $\mcC_{\ad}$, and $Y_{i}=g X_{i}$ are the 2-dimensional simple
    objects in $\mcC_{g^{2}}$. Then
    $Y_{i}^{*}=g^{3}X_{i}^{*}=g\(g^{2}X_{i}^{*}\)=gX_{i}^{*}$. So it suffices to
    show that $X_{i}$ are self-dual. To this end, observe that $X_{i}\otimes
    X_{i}^{*}=\1\oplus g^{2}\oplus \tilde{X_{i}}$ for some 2-dimensional
    simple $\tilde{X_{i}}$. In particular, $\1$ and $g^{2}$ are simples in
    $\(\mcC_{\ad}\)_{\ad}$ and all 2-dimensionals in $\(\mcC_{\ad}\)_{\ad}$ are
    self-dual. Now suppose $\(\mcC_{\ad}\)_{\ad}\neq\mcC_{\ad}$. Then
    $\mcC_{\ad}$ has a nontrivial universal grading, the trivial component is
    $\(\mcC_{\ad}\)_{\ad}$ and has dimension $2+4k$ for some integer $k$. The
    remaining components must also have this dimension, but consist entirely of
    2-dimensional objects. This is not possible and so
    $\(\mcC_{\ad}\)_{\ad}=\mcC_{\ad}$.
  \end{proof}
%
%

  \begin{lemma}
    \label{Lemma: Cyclically Generated}
    $\mcC_{\intcat}$ is Grothendieck equivalent to
    $\Rep\(\mbbZ_{m}\rtimes\mbbZ_{4}\)$ and $\mcC_{\ad}$ is Grothendieck
    equivalent to $\Rep\(D_{2m}\)$.
  \end{lemma}
  \begin{proof}
    Since $\mcC_{ad}$ is a subcategory of $\mcC_{\intcat}$, then
    $C_{\mcC_{\intcat}}\paren{\mcC_{\intcat}}\subset
    C_{\mcC_{\intcat}}\paren{\mcC_{ad}}=\Rep\mbbZ_{2}$. Given that $\mcD$ is prime, then $\mcC_{\intcat}$ is not modular and so its M\"{u}ger center is $\mcC_{\intcat}\cong\Rep\mbbZ_{2}$ Tannakian. Thus
    $\paren{\mcC_{\intcat}}_{\mbbZ_{2}}$ is pointed, modular, and of dimension
    $2m$. In particular, $\paren{\mcC_{\intcat}}_{\mbbZ_{2}}$ is cycically
    generated, \textit{i.e}, the category is tensor generated by a single
    object. So by \cite[Proposition 4.30(i)]{DGNO1} $\mcC_{\intcat}$ must be
    cyclically generated.
    The remainder of the statement follows immediately from
    \cite[Theorem 4.2 and Remark 4.4]{NR1} and \lemmaref{Lemma: self-dual}.
  \end{proof}
%

  To completely determine the Grothendieck class of $\mcC$ we must determine
  $X_{i}\otimes V_{k}$ and $g\otimes V_{k}$ where $V_{k}$ are the four
  non-integral objects.
  \begin{cor}
    \label{Cor: non-integral objects}
    $\mcC$ has four non-integral objects: $V$, $gV$, $g^{2}V$, $g^{3}V$.
    Moreover, $V^{*}=g^{3}V$. In particular, the GN-grading of $\mcC$ is
    $\mbbZ_{2}$ and all non-integral objects have dimension $\sqrt{m}$.
  \end{cor}
  \begin{proof}
    We know that the GN-grading is given by a nontrivial abelian 2-group and
    hence is $\mbbZ/2\mbbZ$, by \lemmaref{Lemma: pkm} and \corref{Z_4}. In
    particular, since $m$ is square-free, there is a square-free integer $x$
    such that the non-integral objects have dimension $\sqrt{x}$ and
    $2\sqrt{x}$.

    A straightforward application of the fusion symmetries and a dimension
    calculation reveals that $g^{2}$ fixes all simples of dimension $2\sqrt{x}$.
    Using this fact, we can appeal to the balancing equation
    and find that $S_{g^{2},g^{2}X} = 2\sqrt{x}$ and
    $S_{g^{2},g^{2}Y}=\frac{\th_{Y}}{\th_{g^{2}Y}}\sqrt{x}$ for any simples $X$
    and $Y$ of dimension $2\sqrt{x}$ and $\sqrt{x}$ respectively. However,
    $g^{2}$ is self-dual and so $\th_{Y}/\th_{g^{2}Y}=\pm1$.  Now observe that
    the orthogonality of the $g^{2}$ and $\1$ columns of the $S$-matrix can
    only be satisfied if $\th_{Y}=-\th_{g^{2}Y}$ and there are no objects of
    dimension $2\sqrt{x}$. In particular, $g^2$ moves all simples of dimension
    $\sqrt{x}$.

    Next, let $V,W\in\mcC_{g}$ be simples of dimension $\sqrt{x}$. Moreover,
    without loss of generality we may assume that $g$ is a subobject of $V\otimes
    V$. Then $\1$ is a subobject of $\(g^{3}V\)\otimes V$ and hence $V^{*} =
    g^{3}V$.  Next note that by the universal grading and the parity of $x$ we
    can deduce that either $g$ or $g^{3}$ is a subobject of $V\otimes W$ (but
    not both). In the former case $\1$ is a subobject of $V\otimes g^{3}W$
    and in the later case $\1$ is a subobject of $V\otimes gW$. Thus
    $g^{3}W=V^{*}=g^{3}V$ or $gW=V^{*}=g^{3}V$, according to the invertible
    subobject appearing in $V\otimes W$. Since $W$ was arbitrary we can conclude
    that there are exactly four non-integral objects: $V,gV,g^{2}V$, and
    $g^{3}V$. Moreover, each of these objects has dimension $\sqrt{m}$.
  \end{proof}
  \begin{lemma}
    \label{Lemma: 2 with sqrtm}
    $X_{i}\otimes V\cong V\oplus g^{2}V$
  \end{lemma}
  \begin{proof}
    By dimension count and the grading we have $X_{i}\otimes V=V\oplus g^{2}V$,
    $2V$, or $2g^{2}V$. The later two are not possible as $g^{2}$ fixes $X_{i}$
    and hence must fix $X_{i}\otimes V$.
  \end{proof}

%
  \begin{theorem}
    \label{Theorem: p3m main result}
    If $\mcC$ is a non-pointed modular category of dimension $p^{3}m$ where $p$
    is a prime and $m$ is a square-free integer that is coprime to $p$,
    then, $p=2$ and one of the following is
    true:
    \begin{itemize}
      \item[(i)] $\mcC$ is a Deligne product of an even metaplectic modular
      category of dimension $8\ell$ and a pointed $\mbbZ_k$-cyclic modular
      category with $k$ odd, or
      \item[(ii)] $\mcC$ is the Deligne product of a Semion modular category with
      a modular category of dimension $4m$ (see \cite{BGNPRW1}).
    \end{itemize}
  \end{theorem}
  \begin{proof}
    By Lemmas \ref{Lemma: 4 divides} and \ref{Lemma: integral p3m}, and
    \propref{prop: prime or classified}, \corref{Cor: Dimension 8 SWI Fusion
    Contains Ising} it suffices to consider $p=2$, $m$ odd and square-free, and
    $\mcC$ prime. In this case we may apply Lemmas \ref{Lemma: 2 with sqrtm},
    \ref{Lemma: Cyclically Generated}, and \corref{Cor: non-integral
    objects} to conclude that $\mcC$ is metaplectic.
  \end{proof}
%

  \printbibliography

\end{document}